\documentclass[reqno, 11pt]{amsart}
\usepackage{mathdots}
\usepackage{amssymb}
\usepackage{euscript}

\usepackage{array}

\newtheorem{Thm}[equation]{Theorem}
\newtheorem{Lem}[equation]{Lemma}

\newtheorem{Prop}[equation]{Proposition}

\theoremstyle{remark}
\newtheorem{Rem}[equation]{Remark}
\theoremstyle{remark}

\theoremstyle{definition}
\newtheorem{Def}[equation]{Definition}

\numberwithin{equation}{section}

\newcommand{\R}{\mathrm{\bf R}}           
\newcommand{\C}{\mathrm{\bf C}}           

\newcommand{\fa}{{\mathfrak a}}             

\newcommand{\fg}{{\mathfrak g}}
\newcommand{\fh}{{\mathfrak h}}
\newcommand{\fk}{{\mathfrak k}}
\newcommand{\fl}{{\mathfrak l}}

\newcommand{\fp}{{\mathfrak p}}
\newcommand{\fq}{{\mathfrak q}}

\newcommand{\ft}{{\mathfrak t}}
\newcommand{\fu}{{\mathfrak u}}

\newcommand{\gt}{\theta}
\newcommand{\gT}{\Theta}

\newcommand{\Cal}{\mathcal}

\renewcommand{\bar}[1]{\overline{#1}}

\begin{document}
\parskip=5pt
\baselineskip= 14pt

\title[Unitary Globalization]
{ On the Unitary Globalization of Cohomologically Induced Modules}
\author{L. Barchini}
\address{Oklahoma State University\\
Mathematics Department\\
    Stillwater, Oklahoma 74078}
\email {leticia@math.okstate.edu}
\author{Petr Somberg}
\address{Mathematical Institute,\\
Charles University,\\
 Sokolovsk\'a 83, Prague\\
 Czech Republic}
\email{somberg@karlin.mff.cuni.cz}
\begin{abstract}
We describe the unitary globalization
of cohomologically induced modules $A_{\fq}(\lambda)$.
The purpose of the paper is to give a geometric realization
of the unitarizable modules. Our results do not constitute a proof of unitarity.
\end{abstract}

\maketitle

\section{Introduction} 

The orbit method suggests a close connection between 
irreducible representations of a Lie group $G_\R$ and co-adjoint orbits.
In the case of nilpotent groups, unitary representations correspond to co-adjoint orbits.
Kirillov used such correspondence to geometrically construct unitary representations of nilpotent groups.
When $G_\R$ is a  real reductive group, attached to elliptic co-adjoint orbits is a family of
irreducible unitarizable $(\fg, K)$-modules, $A_{\fq}(\lambda)$. These are  called cohomologically induced modules. The purpose of this paper is to  give a geometric description of the corresponding $G_\R$ unitary representations.

The representation theory of $G_\R$ is more subtle than that of its Lie algebra. According to a theorem of
Harish-Chandra, the space of $K$-finite vectors of an admissible irreducible $G_\R$-representation is
an irreducible $(\fg, K)$-module. This relationship between $G_\R$ and $(\fg, K)$ modules is not bijective.
The subtlety comes from the fact that different topologies yield different $G_\R$-modules. It is known, from
the work of Casselman, Wallach \cite{Casselman89} and Schmid \cite{Schmid85}, that each admissible irreducible 
$(\fg, K)$-module admits a maximal globalization (resp. minimal globalization) to a $G_\R$ module over a
Frechet space, $\Cal F$ (resp. to a $G_\R$ module over the topological dual of $\Cal F$).  When the $(\fg, K)$-module is an 
$A_{\fq}(\lambda)$,
the maximal globalization (minimal globalization)  corresponds to a Dolbeault cohomology representation, $H^{n,s}(G_\R/L_\R, \Cal L_{\lambda})$,
(corresponds to a compactly supported cohomology space). The unitary globalization should correspond to
a topological space ``in between'' the minimal and the maximal globalization.  This is the space we want to describe.

 In \cite{Vogan84}, Vogan  proposed to (a) explicitly describe the $G_\R$-invariant Hermitian form
on the minimal globalization of $A_{\fq}(\lambda)$ (identified as a space of compactly supported cohomology)
and (b) describe the unitary globalization as the completion of the compactly supported cohomology
with respect to the Hermitian form given in (a). This proposal amounts to finding an explicit $G_\R$-intertwining map
from the Hermitian dual $H^{n,s}(G_\R/L_\R, \Cal L_{\lambda})^h$
 to  $H^{n,s}(G_\R/L_\R, \Cal L_{\lambda})$.  Moreover, such map
 is a ``kernel type '' transform. Indeed, if $(G_\R/ L_\R)^{\text{opp}}$ is
 $G_\R/L_\R$ endowed with the opposite complex structure, then
 the kernel is determined by a $\text{diag }(G_\R\times G_\R)$-invariant class
$[w]$  in $H^{(n,n)(s,s)}(G_\R/L_\R\times (G_\R/L_\R)^{\text{opp}}, \Cal L_{\lambda}\otimes \Cal L_{-\lambda})$.
 The problem is to describe such a cohomology class.
 
 Under some positivity assumptions on $\lambda$, $A_{\fq}(\lambda)$ can be also realized  as the space of 
 $K$-finite 
 solutions  of an elliptic differential operator $\Cal D$, acting on sections of a bundle over
 $G_\R/K_\R$.  This is indeed the content of ([\cite{Wong95}, Corollary 38 and Section 7]).
 The geometric construction of Dolbeault cohomology is related to the solution space of $\Cal D$
 via the so-called Real Penrose transform.
 Moreover, the Real Penrose transform determines an isomorphism between 
 the cohomology realization and $\mathop{Ker } \Cal D$.  See ([\cite{Wong95}, Section 7]),
 ([\cite{BKZ}, Section 10]) and  \cite{B95}. In this paper we determine
 (a) the hermitian dual $(\mathop{Ker } \Cal D\big)^h$, (b) the space of continuous $G_\R$-intertwining maps
 from $(\mathop{Ker } \Cal D\big)^h$ to  $\mathop{Ker } \Cal D$. Such intertwining maps are also of
 ``kernel type". We describe the kernels in terms of  generalized spherical functions, $F$.
 In particular, when $A_{\fq}(\lambda)$ is the Harish-Chandra module of a representation in the discrete series,
$F$ is the function defined by Flensted-Jensen in \cite{Flensted-Jensen80}. The function $F$, given in
Theorem \ref{main}, depends solely on the minimal $K$-type of $A_{\fq}(\lambda)$. This is consistent with the fact
that $A_{\fq}(\lambda)$ is unitary if and only if the Hermitian form is definite on its bottom layer, see \cite{Vogan84}.
 The cohomology class $[w]$ is completely determined by $F$.

.

\section{The maximal globalization of $A_{\fq}(\lambda)$}\label{sec: DC}

\subsection{Dolbeault Cohomology}  We recall  
results  from \cite{Vogan84} and \cite{Wong95} that will be relevant to our work.
 Our underlying group  $G_\R$ is assumed to be  real reductive  with complexification $G$ and Cartan involution $\gT$.
We let $K_\R$ be the fixed point group of $\gT$, the maximal compact subgroup.
We denote the Lie algebra of Lie groups by $\fg_\R$, $\fk_\R$ etc., and their complexifications by $\fg$, $\fk$, etc.
Letting $\gt$ be the differential of $\gT$ we write the decomposition of  $\fg$ into $\pm 1$ eigenspaces
as $\fg = \fp \oplus \fk$. We choose a Cartan  subalgebra $\ft \subset \fk$ and extend it to a Cartan subalgebra $\fh = \ft \oplus \fa$
of $\fg$. Using the Killing form, $B(\cdot, \cdot)$, we consider
$\ft^*\subset \fh^*\subset \fg^*$. Then an element $\lambda \in \ft^*$ is elliptic, and the orbit
$G_\R\cdot \lambda \subset \fg^*$ is an elliptic co-adjoint orbit.  We may identify this orbit with the homogeneous
space $G_\R/L_\R$, where $L_\R$ is the centralizer in $G_\R$ of $\lambda$.
On the other hand, $\lambda$ defines a $\gt$-stable parabolic
subalgebra of $\fg$ as follows. Denote by $\Delta = \Delta(\fg, \fh)$   the roots of  $\fh$ in $\fg$. Then the
parabolic subalgebra associated to $\lambda$ is 
$\fq = \fl \oplus \fu$,
where the Levi factor  $\fl$ is spanned by $\fh$ and all root spaces $\fg^{\alpha}$
with $\langle \lambda,  \alpha \rangle = 0$, and $\fu$ is spanned
by all root spaces $\fg^{\alpha}$
 with $\langle \lambda, \alpha \rangle  > 0$. 
If $\Cal Q$ is  the normalizer of $\fq$ in $G$, one
sees that $ L_\R = \Cal Q \cap G_\R$, so $G_\R/L_\R$ embeds into the generalized flag variety
$G/\Cal Q$
as an open subset. In particular, $G_\R/L_\R$ has a $G_\R$-invariant complex structure; the
antiholomorphic tangent space at the identity coset is naturally identified with $\fg/\fq \simeq \fu$.
A similar construction makes $K_\R/(K_\R\cap L_\R)$ into a complex compact submanifold of $G_\R/L_\R$.

 Observe that each $\theta$-stable
parabolic subalgebra  $\fq $  containing $\fl$ gives a complex structure on
$G_\R/L_\R$; these are in fact all different. In the language of geometric quantization, these
parabolic subalgebras are the invariant complex polarizations at $\lambda$. Typically, in geometric
quantization, one chooses a particular polarization and this is what we will do here.

To attach a representation to $G_\R \cdot \lambda$, we assume that $\lambda$ lifts to a character $\chi_{\lambda}$ of
$L_\R$. Then there is  a holomorphic homogeneous line bundle associated to $\chi_{\lambda}$. If $n $ is the complex dimension of $G_\R/L_\R$, 
it is natural  to attach cohomology representations $H^{n,p}(G_\R/L_\R, \Cal L_{\lambda})$ to the orbit
$G_\R \cdot \lambda$.  The Dolbeault cohomology  $H^{n,p}(G_\R/L_\R, \Cal L_{\lambda})$ can be computed, for
example, by Leray covers, by $C^{\infty}$-differential forms or by currents (differential forms with distribution
coefficients). All these approaches yield the same cohomology groups, as vector spaces. Indeed, more is true.
One can define strong topologies on $H^{n,p}(G_\R/L_\R, \Cal L_{\lambda})$ with respect to Leray covers,
as well as with respect to $C^{\infty}$-forms or currents, see \cite{Laufer67}. 
(For example, the strong topology with respect
to  $C^{\infty}$-forms is given by uniform convergence on compact sets for all derivatives of the coefficients when written in terms of coordinates in the local charts.) 

\begin{Thm} [\cite{Laufer67}, Theorem 2.1, Theorem 3.2] \label{Laufer} The strong topologies on\newline
$H^{n,p}(G_\R/L_\R, \Cal L_{\lambda})$ with respect to Leray covers, $C^{\infty}$-forms and currents
all coincide.
\end{Thm}

It is not {\it  a priori } clear that  the natural action of $G_\R$ (by left translation)  on $H^{n,p}(G_\R/L_\R, \Cal L_{\lambda})$  is continuous in the topology of Theorem \ref{Laufer}. The difficulty is that 
a topology on $H^{n,p}(G_\R/L_\R, \Cal L_{\lambda})$ is Hausdorff only if the operator $\bar{\partial}$ that defines
the cohomology space has the closed range property in that given topology.
This delicate issue was settled by Wong  in \cite{Wong95} and \cite{Wong99}.  
Wong proved that (a)  when $p = s = \text{dim} (K_\R/ (K_\R\cap L_\R)$   $H^{n,p}(G_\R/L_\R, \Cal L_{\lambda})$ is a non-zero continuous Fr\'echet
representation (by showing that the image
of $\bar{\partial}$ is closed in the $C^{\infty}$-topology on forms), (b)
 $H^{n,s}(G_\R/L_\R, \Cal L_{\lambda})_{K-\text{finite}}$ is a cohomologically induced  $(\fg, K)$-module,
 (c) $H^{n,s}(G_\R/L_\R, \Cal L_{\lambda})$ is the maximal globalization of its underlying 
 Harish-Chandra module in the sense of \cite{Schmid85}.
 
 \begin{Rem} It is important to note that the results in \cite{Wong95} and \cite{Wong99} are more general
 than those stated above. On the one hand, \cite{Wong99} allows the inducing bundle  to be infinite dimensional. On the other hand,  the conditions on the line bundle in \cite{Wong95}
 are less restrictive that the ones used here. Indeed, for a fixed  positive system $\Delta^+(\fg, \fh)$ 
 that contains $\Delta(\fu)$ the main theorem in \cite{Wong95} holds for   a one-dimensional representation $\chi_{\nu}$
 of $L_R$ with  weight $\nu$, 
with  $\langle \nu +\rho , \alpha \rangle > 0$ for roots $\alpha$ 
in $\fu$ (and  $\rho$
equal to  the half the sum of  positive roots).  In this paper  we assume that
 \begin{equation}\label{condition} 
 \langle \lambda  , \alpha \rangle > 0  \text{ for all root } \alpha\in \Delta(\fu).
 \end{equation}
 We impose  this, more restrictive, assumption on $\lambda$
 in order to have control on the $K$-type structure of 
 $H^{n,s}(G_\R/L_\R, \Cal L_{\lambda})_{K-\text{finite}}$, see \cite{VZucker84}.
\end{Rem}

We keep the notation $\fq = \fl \oplus \fu$ and write $\fu = \fu\cap \fp \oplus \fu \cap\fk$.
We let $\Delta(\fu\cap\fp)$ stand for the set of weights in $\fu\cap\fp$ with respect to $\fh$
and we write $\rho(\fu\cap\fp)$ for half the sum of the positive members of $(\fu\cap\fp)$ with respect to $\fh$.

\begin{Thm}  \label{Wong}
Suppose that $\fq = \fl \oplus \fu$ is a $\gt$-stable parabolic subalgebra of $\fg$ and
let $\fh = \ft \oplus \fa \subset \fl$ be a Cartan subalgebra. Let $\lambda \in \ft^*$ be an integral weight and
assume that
\begin{equation*}
\langle \lambda, \alpha \rangle > 0, \text{ for all root } \alpha \in \Delta(\fu).
\end{equation*}
Identify the elliptic co-adjoint orbit $G_\R \cdot \lambda$ with the homogeneous space
$G_\R/L_\R$. Endow $G_\R/L_\R$ with the complex structure  so that  the antiholomorphic tangent space
at the identity is identified with $\fu$.  Let $s = \text{dim}_\C(K_\R/(L_\R\cap K_\R))$. Then,
\begin{enumerate}
\item The strong topology of Theorem \ref{Laufer} on $H^{n,s}(G_\R/L_\R, \Cal L_{\lambda})$
is Hausdorff. In particular, the $\bar{\partial}$ Dolbeault operator has closed range.
\item $H^{n,p}(G_\R/L_\R, \Cal L_{\lambda}) = 0$, unless $p = s$.
\item The continuous representation of $G_\R$ on  $H^{n,s}(G_\R/L_\R, \Cal L_{\lambda})$ is irreducible
and Hermitian.
It is the maximal globalization of the underlying $(\fg, K)$-module.
\item $(\pi_{\lambda}, H^{n,s}(G_\R/L_\R, \Cal L_{\lambda}))$ contains with multiplicity one,
the $K$-type with highest weight 
\begin{equation*}
\mu = \lambda + 2 \rho(\fu \cap \fp).
\end{equation*}
If $\mu'$ is the highest weight of a $K$-type occurring in $\pi_{\lambda}\vert_K$, then
$\mu'$ is of the form
\begin{equation*}
\mu' = \mu + \sum n_{\alpha} \, \alpha, \text{ with } n_{\alpha} \in \mathbb{N}  \text{   and } \alpha \in \Delta(\fu\cap\fp).
\end{equation*}
\end{enumerate}
\end{Thm}

\begin{proof} Part (1) is proved in \cite{Wong95}. Parts (2) and (3) are results in  \cite{Vogan84},
written here in the language of Dolbeault cohomology. Part (4) is Theorem 5.3 in \cite{VZucker84}.
\end{proof}

\subsection{The kernel of Schmid's $\Cal D$-differential operator}

The maximal globalization of $A_{\fq}(\lambda)$ can be described
as the solution space of an elliptic operator acting on the space of smooth sections of a bundle
over $G_\R/K_\R$.
Let $(\tau_{\mu}, V_{\mu})$ be the minimal $K$-type occurring in  
$H^{n,s}(G_\R/L_\R, \Cal L_{\lambda})_{\text{$K$-finite}}$.
The relevant  vector bundle over $G_\R/K_\R$ is the one  induced by $V_{\mu}$. Indeed,
such realization plays a key role in proving that the   Dolbeault cohomology 
endowed with the topology of Theorem \ref{Laufer} is Hausdorff  (\cite{Schmid67} and \cite{Wong95}). This alternative realization of cohomologically
induced modules will  be important in our work.
We start by recalling   that  $ \mu =\lambda + 2 \rho(\fu \cap \fp)$ and that the highest weights of the
irreducible $K$-modules  occuring in $V_{\mu}\otimes \fp$
are of the form $\mu + \alpha$ with $\alpha \in \Delta(\fp, \ft)$. Following \cite{Wong95} we introduce the following definition.

\begin{Def} Let 
\begin{equation*}
V_{\mu}^- = \sum_{\beta \in \Delta(\fu \cap \fp)}\tau_{\mu-\beta} \subset V_{\mu}\otimes\fp
\end{equation*}
and let
$\mathbb P : V_{\mu}\otimes\fp \to V_{\mu}^-$ be the canonical projection. Choose
$\{ X_i\}$ an orthonormal basis of $\fp$ with respect to $( U, V) = - B( U, \bar{\gt V})$. For
$F \in C^{\infty} (G_\R/K_\R, V_{\mu})$ define
\begin{equation}\label{schmid}
\Cal D F (g ) = \sum_{i} \mathbb P \big[ (X_i F) ( g) \otimes {\bar{X}_i} \big].
\end{equation}
\end{Def}

\begin{Prop} (\cite{Wong95}, Proposition 49)
$\Cal D$ is well defined (i.e. independent of the choice of basis). If
$\lambda$ is sufficiently positive, $\Cal D$ is an elliptic operator.
\end{Prop}

Observe that  $C^{\infty} (G_\R/K_\R, V_{\mu})$, endowed with the topology of uniform convergence
over compact subsets of functions and their derivatives, is a Fr\'echet space.
The operator $\Cal D : C^{\infty} (G_\R/K_\R, V_{\mu})\to C^{\infty} (G_\R/K_\R, V^-_{\mu})$ is continuous.
Hence, the  kernel space $\mathop{Ker} \Cal D$ is closed in $C^{\infty} (G_\R/K_\R, V_{\mu})$ and it inherits the structure
of Fr\'echet space.  In order to emphasize the space on which $\Cal D$ acts we write
\begin{equation}\label{ref}
\mathop{Ker} \Cal D = C^{\infty}_{\Cal D} (G_\R/K_\R, V_{\mu}). 
\end{equation}

\begin{Thm}  \label{Penrose}  Keep the assumptions on Theorem \ref{Wong}.
If $\langle \lambda - 2 \rho(\fl\cap\fk), \alpha\rangle > 0$ for all root
$\alpha \in \Delta(\fu)$, then there exists a  $G_\R$-equivariant map
\begin{equation*}
{\Cal P} : H^{n,s}(G_\R/L_\R, \Cal L_{\lambda})\to  C^{\infty}_{\Cal D} (G_\R/K_\R, V_{\mu}). 
\end{equation*}
The map $\Cal P$ is an homeomorphism of topological spaces.
\end{Thm}

\begin{proof}

When  $\text{rank} (G_\R) = \text{rank}  (K_\R)$ and $L_\R = T_\R$ is a maximal compact torus
$\rho(\fl\cap\fk) = 0$ and the Theorem holds for all $\lambda$ satisfying the positivity condition
of Theorem \ref{Wong}. This result is proved in \cite{Schmid75}.
The general statement follows form  (\cite{Wong95}, Corollary 38 and the proof of Proposition 52).
Let $\Cal N$ be the normal  bundle of the compact submanifold
$K_\R/(L_\R\cap K_\R)$ in $G_\R/L_\R$  and let $(k)$ signify the $k$-th symmetric power.
By (\cite{Wong95}, Corollary 38) the Theorem  holds if
\begin{equation}\label{vanishing}
H^i (K_\R/(L_\R\cap K_\R), \Cal L_{\lambda+ 2 \rho(\fu)}\otimes (\Cal N^*)^{(k)} ) = 0
\text{ for all }  i < s \text{ and all } k\geq 0.
\end{equation}
The vanishing condition (\ref{vanishing}) follows from (\cite{Griffiths69}, Theorem G).
Our assumption on $\lambda$ guarantees that the hypothesis of  (\cite{Griffiths69}, Theorem G)  are satisfied.

\end{proof}

\begin{Rem} The map $\Cal P$ is the Real Penrose transform in (\cite{Schmid67}, Lemma 7.1)
 (\cite{Wong95}, Section 7), (\cite{BKZ}, \cite{B95}) and  (\cite{Zierau}, Lecture 7) .
  The transform 
$${\Cal P} : H^{n,s}(G_\R/L_\R, \Cal L_{\lambda})\to  C^{\infty}_{\Cal D} (G_\R/K_\R, V_{\mu})$$
is an integral transform; see \cite{BKZ}. If $\omega_c \in \wedge^s(\fu\cap\fk)^*$ is a normalized top form, then
\begin{equation}\label{pa}
\Cal P ( w) (x) = \int_{K/(L\cap K)} \tau_{\mu}(k) v_{\mu} \; \langle w (x k), 1_{\lambda}\otimes \omega_c\rangle\; dk,
\end{equation} 
with $v_{\mu}$ a normalized highest weight vector in $V_{\mu}$.

Under the positivity assumptions of Theorem \ref{Penrose}  the transform $\Cal P$ is injective onto
$\mathop{Ker} \Cal D$. Each $G \in \mathop{Ker} \Cal D$ determines a unique cohomology class $[\eta_G]$.
One way to determine a representative of  $[\eta_G]$ is to follow the recursive procedure described in
 (\cite{Schmid67}, Lemma 7.1) while keeping track of the required $\fl$-equivariant property.

\end{Rem}


\section{The minimal globalization of $A_{\fq}(\lambda)$}

\subsection{Compactly supported cohomology} 
Theorem \ref{Wong}   identifies the maximal globalization of cohomologically-induced modules as
Dolbeault cohomology  representations.  In this section we summarize relevant information on
the minimal globalization of  $H^{n,s}(G_\R/L_\R, \Cal L_{\lambda})_{\text{$K$-finite}}$.
 
It is known that the minimal globalization
$H^{n,s}(G_\R/L_\R, \Cal L_{\lambda})_{\text{$K$-finite}}$ is the topological dual of its maximal globalization.
Since we know that $\bar{\partial}$ has the closed range property, Serre duality implies that the
minimal globalization occurs as compactly supported cohomology.
 As for Dolbeault cohomology, compactly supported cohomology can be calculated  in different ways.
In particular, such cohomology groups can be computed by
using Leray covers, $C^{\infty}$ compactly supported forms or compactly supported currents
(forms with compactly supported distribution coefficients). All these approaches yield the same cohomology group
as vector spaces. These cohomology spaces can be endowed with strong topologies
as described in \cite{Laufer67}. Theorem \ref{Laufer} holds for compactly supported cohomology.

\begin{Thm} (\cite{Laufer67}, Theorem 2.1).\label{Laufer2} The strong topologies on $H^{0,q}_c(G_\R,L_\R, \Cal L_{\lambda})$ with respect to
Leray covers, $C^{\infty}_c$-forms and compactly supported currents coincide.
\end{Thm}

\begin{Thm} In the setting of Theorem \ref{Wong} and 
Theorem \ref{Laufer2},  write $\Cal L_{\lambda}^*$ for the bundle dual to $\Cal L_{\lambda}$.

\begin{enumerate}
\item There is a natural topological isomorphism 
\begin{equation*}
H_c^{0,n-s}(G_\R/L_\R, \Cal L_{\lambda}^*)^* \simeq H^{n,s}(G_\R/L_\R, \Cal L_{\lambda}).
\end{equation*}
\item The topological space $H_c^{0,n-s}(G_\R/L_\R, \Cal L_{\lambda}^*)$ is Hausdorff.
\end{enumerate}
\end{Thm}

\begin{proof} A more general version  of this duality theorem is proved in (\cite{Laufer67}, Theorem 3.2).
Also see \cite{Serre55} and the remarks in (\cite{Vogan08}, page 67).
\end{proof}

\begin{Thm}[\cite{Bratten97}, page 285; \cite{Vogan08}, Corollary 8.15] \label{compact} Keep the hypothesis  of Theorem \ref{Wong}.
\begin{enumerate}
\item $H_c^{0,q}(G_\R/L_\R, \Cal L_{\lambda}^*)= 0$, unless $q = n-s$.
\item $H_c^{0,n-s}(G_\R/L_\R, \Cal L_{\lambda}^*)$ is non-zero and it admits a continuous $G_\R$ action.
The resulting representation is irreducible. It is the minimal globalization of the underlying $(\fg, K)$-module.
\item $H_c^{0,n-s}(G_\R/L_\R, \Cal L_{\lambda}^*)$ admits an invariant Hermitian form.
\end{enumerate}
\end{Thm}

\subsection{The  topological dual of $\mathop{Ker } \Cal D$}
\begin{Prop}
Endow  $C^{\infty}_{\Cal D} (G_\R/K_\R, V_{\mu})$ with the strong topology relative to the smooth
topology on $C^{\infty} (G_\R/K_\R, V_{\mu})$.
Write 
\begin{equation*}
 \big( \mathop{Ker}\Cal D\big)^{\perp}=
\{ \Lambda \in C^{\infty} (G_\R/K_\R, V_{\mu})^*  \vert\; \Lambda\vert_{\mathop{Ker }\Cal D} = 0\}.
\end{equation*}
\begin{enumerate}
\item  
  $C^{\infty} (G_\R/K_\R, V_{\mu})^*/ \big(\mathop{Ker }\Cal D\big)^{\perp},$
  endowed with the quotient topology of the strong topology on $C^{\infty} (G_\R/K_\R, V_{\mu})^*$, is homeomorphic to
  $C^{\infty}_{\Cal D} (G_\R/K_\R, V_{\mu})^*$ endowed with the strong topology.
\item The topological spaces  $C^{\infty}_c (G_\R/K_\R, V_{\mu})/ \big( \mathop{Ker }\Cal D\big)^{\perp}\cap C^{\infty}_c (G_\R/K_\R, V_{\mu}) $ and 
$C^{\infty}_{\Cal D} (G_\R/K_\R, V_{\mu})^*$ and  are homeomorphic.
\end{enumerate}
\end{Prop}

\begin{proof}
The first statement of the Proposition follows from (\cite{Vogan08}, Prop.8.8 (2)).
Indeed, the space of smooth sections of a finite-dimensional vector bundle, when endowed with the
smooth topology, is a Fr\'echet nuclear space. Thus,  it is reflexive.
Also, $\mathop{Ker }\Cal D \subset C^{\infty}(G_\R/K_\R, V_{\mu})$ is a closed subspace, as $\Cal D$
is a continuous operator. Hence, the hypothesis of (\cite{Vogan08}, Prop.8.8 (2)) is satisfied. We conclude
that 
 if $i: \mathop{Ker }\Cal D \to C^{\infty}(G_\R/K_\R, V_{\mu})$ is the natural inclusion, then  the transpose map $i^{t}$ induces the desired homeomorphism. This is,
\begin{equation*}
i^{t}: C^{\infty} (G_\R/K_\R, V_{\mu})^*/ \big( \mathop{Ker }\Cal D\big)^{\perp}\to C^{\infty}_{\Cal D} (G_\R/K_\R, V_{\mu})^*,
\end{equation*}
is a homeomorphism of topological spaces.

In order  to prove  the second statement of the Proposition
observe that \newline
 $C^{\infty} (G_\R/K_\R, V_{\mu})^*/ \big( \mathop{Ker }\Cal D\big)^{\perp}$
 endowed with the quotient topology of the strong topology on $C^{\infty} (G_\R/K_\R, V_{\mu})^*$
 is Hausdorff. Indeed, the transpose  $\Cal P^{t}$ of the homeomorphism $\Cal P$ in
 Theorem (\ref{Penrose}) is a homeomorphism from  $\big(\mathop{Ker }\Cal D\big)^*$ to the compactly supported cohomology $H_c^{0,n-s}(G_\R/L_\R, \Cal L^*_{\lambda})$.  As the cohomology space is Hausdorff (it is the minimal globalization
 of its underlying $(\fg, K)$-modules), $\big(\mathop{Ker}\Cal D\big)^* \simeq  C^{\infty} (G_\R/K_\R, V_{\mu})^*/ \big( \mathop{Ker }\Cal D\big)^{\perp} $ is Hausdorff.
It follows, see (\cite{Treves67}, Chapter 4),  that the kernel of the continuous map
 $C^{\infty}_c (G_\R/K_\R, V_{\mu})\to C^{\infty} (G_\R/K_\R, V_{\mu})^*/ \big( \mathop{Ker }\Cal D\big)^{\perp}$
 is closed. Hence, we have a continuous $G_\R$-equivariant map 
\begin{equation}\label{last}
T: C^{\infty}_c (G_\R/K_\R, V_{\mu})/ \big( \mathop{Ker }\Cal D\big)^{\perp} \to
C^{\infty} (G_\R/K_\R, V_{\mu})^*/ \big( \mathop{Ker } \Cal D\big)^{\perp},
\end{equation} 
where $C^{\infty} (G_\R/K_\R, V_{\mu})^*/ \big( \mathop{Ker } \Cal D\big)^{\perp}$ is the minimal globalization of its
underlying $(\fg, K)$-module, the Proposition follows.
\end{proof}

\section{Hermitian pairings}

If $E$ is a complete locally convex vector space, then its Hermitian dual $E^h$ is given by
\begin{align*}
E^h = \{ \Lambda: E \to \C &\text{ continuous} :\\
 &\Lambda ( a v + b w ) = \bar{a} \Lambda (v) + \bar{b} \Lambda (w)
\text{ for } a, b \in \C \text{ and } v, w \in E\}.
\end{align*}
 The space $E^h$ is conjugate linearly isomorphic to $E^*$. Using this identification
 $E^h$ can be endowed with the strong topology; see  (\cite{Vogan08}, Section 8.3).
 A Hermitian pairing between two complete locally convex vector spaces $E$ and $F$ is
a separately continuous map
\begin{equation*}
\langle \; ,\;  \rangle : E \times F \to \C
\end{equation*}
that is linear in the first variable and  conjugate linear in the second variable.
 This space is 
 in bijection with the space of continuous linear maps
$L(E, F^h)$; see for example (\cite{Vogan08}, Section 9).

\subsection{ Hermitian pairings on $H_c^{0,n-s}(G_\R/L_\R, \Cal L_{\lambda}^*)$}

\begin{Thm} \cite{Vogan08}\label{hdual}
Let $\fq = \fl \oplus \fu$ be a $\gt$-stable parabolic subalgebra. Let $\Cal Q$ be the analytic subgroup of
$G$ with Lie algebra $\fq$. Endow $G_\R/L_\R$ with the complex structure induced by the
open embedding $G_\R/L_\R \subset G/ \Cal Q$.  Let $(G_\R/L_\R)^{\text{opp}}$ be the manifold 
$G_\R/L_\R$ endowed with the opposite complex structure.
\begin{enumerate}
\item The Hermitian dual of $H_c^{0,n-s}(G_\R/L_\R, \Cal L_{-\lambda})$ is
$H^{n,s}((G_\R/L_\R)^{\text{opp}}, \Cal L_{-\lambda})$.
\item The space of separately continuous Hermitian pairings on 
 $H_c^{0,n-s}(G_\R/L_\R, \Cal L_{-\lambda})$
is  isomorphic to $H^{(n,n)(s,s)}(G_\R/L_\R\times (G_\R/L_\R)^{\text{opp}}, \Cal L_{\lambda}\otimes \Cal L_{-\lambda})$.
\item The space of $G_\R$-invariant Hermitian forms on  
$H_c^{0,n-s}(G_\R/L_\R, \Cal L_{-\lambda})$ is the space of $\text{diag}(G_\R\times G_\R)$- invariant
real cohomology classes in\newline
 $H^{(n,n)(s,s)}(G_\R/L_\R\times (G_\R/L_\R)^{\text{opp}}, \Cal L_{\lambda}\otimes \Cal L_{-\lambda})$.
\end{enumerate}
\end{Thm}

\begin{Rem}\label{important}

\begin{enumerate}
\item For a definition of \it{real} cohomology class, see  \rm(\cite{Vogan08}, page 339).\it

\item When $H_c^{0,n-s}(G_\R/L_\R, \Cal L_{-\lambda})$ is computed by the complex of compactly
supported smooth forms, the  isomorphism in part (1) of  Theorem \ref{hdual}  assigns to a
compactly supported form $\phi$ the functional $\Lambda_{\phi}$ 
on $H^{n,s}((G_\R/L_\R)^{\text{opp}}, \Cal L_{-\lambda})$ given by
\begin{align*}
\omega \in H^{n,s}((G_\R/L_\R)^{\text{opp}}, \Cal L_{-\lambda})& \to \Lambda_{\phi}(\omega) \in \C\\
\Lambda_{\phi}(\omega)  & = \int_{G_\R/L_\R} \phi \wedge \sigma(\omega).
\end{align*}

Here $\sigma$ is the  complex conjugation in cohomology induced by the map
\begin{equation}\label{conjugate}
C^{\infty}((G_\R/L_\R)^{\text{opp}}, \wedge^s\bar{\fu}\otimes\wedge^n\fu\otimes \C_{-\lambda})
\to C^{\infty}(G_\R/L_\R, \wedge^s \fu\otimes\wedge^n\bar{\fu}\otimes \C_{\lambda}).
\end{equation}

\item When $H_c^{0,n-s}(G_\R/L_\R, \Cal L_{-\lambda})$ is computed by the complex of compactly
supported smooth forms, and $w (\cdot, \cdot)$ represents a smooth cohomology class
in $H^{(n,n)(s,s)}(G_\R/L_\R\times (G_\R/L_\R)^{\text{opp}}, \Cal L_{\lambda}\otimes \Cal L_{-\lambda})$,
the hermitian pairing   in part (2) of Theorem \ref{hdual}  assigns to a compactly supported form $\phi$
the smooth Dolbeault cohomolgy class represented by $\eta (y) = \int_{G_\R/L_\R} \phi(x) \wedge w(x,y)$.

\item When $H_c^{0,n-s}(G_\R/L_\R, \Cal L_{-\lambda})$ is identified with the space of conjugate linear continuous
maps on $H^{n,s}((G_\R/L_\R)^{\text{opp}}, \Cal L_{-\lambda})$, the Hermitian pairing (2) in
Theorem (\ref{hdual})  assigns to a functional $\Lambda$, the Dolbeault cohomology class
$\eta (y) = \Lambda \big( \sigma\otimes 1) (w(\cdot, y) \big)$. Here $\sigma\otimes 1$ is the conjugation
in (\ref{conjugate}) applied to the ``first variable''.

\item The  space $H^{(n,n)(s,s)}(G_\R/L_\R\times (G_\R/L_\R)^{\text{opp}}, \Cal L_{\lambda}\otimes \Cal L_{-\lambda})$
is topologically isomorphic to the projective tensor\newline
$H^{(n,s}(G_\R/L_\R,  \Cal L_{\lambda})\hat{\otimes}_{\pi} H^{(n,s}((G_\R/L_\R)^{\text{opp}},  \Cal L_{-\lambda})$.
See \rm (\cite{Vogan08}, page 72) and (\cite{Treves67}, Definition 43.2 and 43.5).
\end{enumerate}

\end{Rem}

\subsection{Hermitian pairings in the $G_\R\backslash K_\R$-picture}\label{hp}

Write $(\tau^{\vee}_{\mu}, V^{\vee}_{\mu})$ for  the representation of $K_\R$ contragredient to
$(\tau_{\mu}, V_{\mu})$. Let $T_{\sigma}$ denote the conjugate linear isomorphism
from $V_{\mu}$ to $V^{\vee}_{\mu}$.

\begin{Def}\label{operators}
For
$F $ a smooth section of the vector bundle  
\begin{equation*}
(G_\R\times G_\R)\times_{K_\R\times K_\R} (V^{\vee}_{\mu}\otimes V_{\mu}),
\end{equation*}
 and $g \in G_\R$ write $(R(1,g)F)(x,y) = F(x, y g)$ and use the same notation for the differential of the right action.
Similarly define $R(g, 1)$. Choose
$\{ X_i\}$ an orthonormal basis of $\fp$ with respect to $( U, V) = - B( U, \bar{\gt V})$.
Let $\mathbb P$ be the projection operator in Definition (\ref{schmid}) and define the differential operator
\begin{equation*}
1\otimes \Cal D: C^{\infty} (G_\R\times G_\R/(K_\R\times K_\R),  V^{\vee}_{\mu}\otimes V_{\mu})
\to C^{\infty} (G_\R\times G_\R/(K_\R\times K_\R),  V^{\vee}_{\mu}\otimes V^-_{\mu})
\end{equation*} 
by means of

\begin{equation*}
[1 \otimes \Cal D F] ( x, y) = \sum_{i}  1\otimes \mathbb P \big[ (R(1, X_i) F) (x,y ) \otimes \bar{X}_i \big].
\end{equation*}
Similarly, define the operator $\underline{\Cal D}\otimes 1$.
\end{Def}

\begin{Def} Let $C^{\infty}_{\underline{\Cal D}\times \Cal D} (G_\R\times G_\R/(K_\R\times K_\R),  V^{\vee}_{\mu}\otimes V_{\mu})$ be the space of smooth sections of the vector bundle
$G_\R\times G_\R\times_{K_\R\times K_\R}  (V^{\vee}_{\mu}\otimes V_{\mu})$ that are annihilated by both
differential operators $1\otimes \Cal D$ and $\underline{\Cal D} \otimes 1$.
Endow the space  $C^{\infty}_{\underline{\Cal D}\times \Cal D} (G_\R\times G_\R/(K_\R\times K_\R),  V^{\vee}_{\mu}\otimes V_{\mu})$
with the strong topology relative to the smooth topology on the space of sections.
\end{Def}

\begin{Prop}\label{homeo}
 $H^{(n,n)(s,s)}(G_\R/L_\R\times G_\R/L_\R^{\text{opp}}, \Cal L_{\lambda}\otimes \Cal L_{-\lambda})$
is homeomorphic to  $C^{\infty}_{\underline{\Cal D}\times \Cal D} (G_\R\times G_\R/(K_\R\times K_\R),  V^{\vee}_{\mu}\otimes V_{\mu})$.

\end{Prop}

\begin{proof} 
We first prove that  
 \begin{align}\label{dos}
& C^{\infty}_{\underline{\Cal D}} (G_\R/K_\R ,  V^{\vee}_{\mu}) \hat{\otimes}_{\pi}  C^{\infty}_{\Cal D} (G_\R/K_\R,  V_{\mu}) \text{ is homeomorphic to }\\ \nonumber
 &C^{\infty}_{\underline{\Cal D}\times \Cal D} (G_\R\times G_\R/(K_\R\times K_\R),  V^{\vee}_{\mu}\otimes V_{\mu}).
 \end{align}
 
 It is not difficult to show, using (\cite{Treves67},  Prop 44.1), that  
 $C^{\infty}(G_\R/K_\R,  V_{\mu})\simeq C^{\infty}(G_\R/K_\R)\hat{\otimes}_{\pi}V_{\mu}$.
Using this observation and arguing as in  (\cite{Treves67},  Thm 51.6) we show that
$C^{\infty} (G_\R\times G_\R/(K_\R\times K_\R),  V^{\vee}_{\mu}\otimes V_{\mu})$ is canonically isomorphic
to $C^{\infty} (G_\R/K_\R ,  V^{\vee}_{\mu}) \hat{\otimes}_{\pi}  C^{\infty}(G_\R/K_\R,  V_{\mu})$. $C^{\infty} (G_\R\times G_\R/(K_\R\times K_\R),  V^{\vee}_{\mu}\otimes V_{\mu})$ has the structure of Souslin space, see
(\cite{Treves67}, page 556). As $\text{Ker}(1\otimes \Cal D)$  and  $\text{Ker}(\underline{\Cal D}\otimes 1)$ are closed in
$C^{\infty} (G_\R\times G_\R/(K_\R\times K_\R),  V^{\vee}_{\mu}\otimes V_{\mu})$, we conclude that the space
$C^{\infty}_{\underline{\Cal D}\times \Cal D} (G_\R\times G_\R/(K_\R\times K_\R),  V^{\vee}_{\mu}\otimes V_{\mu})$ is Souslin.
By (\cite{Treves67}, Appendix Corollary 1),  the surjective continuous map
\begin{equation}\label{tres}
C^{\infty}_{\underline{\Cal D}\times \Cal D} (G_\R\times G_\R/(K_\R\times K_\R),  V^{\vee}_{\mu}\otimes V_{\mu}),
\to \text{Ker}(\underline{\Cal D})\hat{\otimes}_{\pi} \text{Ker}(\Cal D)
\end{equation}
is open. This proves our claim.

Next,   we recall  (Remark (\ref{important}), part 5)  that the spaces
\begin{align}\label{uno}
&H^{(n,n)(s,s)}(G_\R/L_\R\times (G_\R/L_\R)^{\text{opp}}, \Cal L_{\lambda}\otimes \Cal L_{-\lambda}),
\text{ and } \\ \nonumber
&H^{n,s}(G_\R/L_\R, \Cal L_{\lambda}) \hat{\otimes}_{\pi} H^{n,s}((G_\R/L_\R)^{\text{opp}}, \Cal L_{-\lambda}).
\end{align}
are homeomorphic.

 To complete the proof of  the Proposition it is enough to argue that the tensor products in displays
  (\ref{uno}) and (\ref{dos})
are homeomorphic.  In order to prove so, recall  that  under our assumptions on $\lambda$,
the Penrose transforms maps  of Theorem (\ref{Penrose}),
\begin{align*}
&\Cal P : H^{n,s}(G_\R/L_\R, \Cal L_{\lambda})\to  C^{\infty}_{\Cal D} (G_\R/K_\R, V_{\mu}) \text{ and }\\
&\Cal P_{\text{opp}} :  H^{n,s}((G_\R/L_\R)^{\text{opp}}, \Cal L_{-\lambda})\to  C^{\infty}_{\underline{\Cal D}} 
(G_\R/K_\R, V^{\vee}_{\mu})
\end{align*}
are homeomorphism of topological spaces.
Since the spaces under consideration are Fr\'echet,  by  (\cite{Treves67}, Prop 43.9), $\Cal P\hat{\otimes} \Cal P_{\text{opp}}$  implements  the desired homeomorphism.
 (For a definition of $\Cal P\hat{\otimes} \Cal P_{\text{opp}}$ see (\cite{Treves67}, Definition 43.6).)

\end{proof}

\begin{Def}

 Write  $\Cal P\otimes \Cal P_{\text{opp}}$ for the map that implements the homeomorphism from
$H^{(n,n)(s,s)}(G_\R/L_\R\times (G_\R/L_\R)^{\text{opp}}, \Cal L_{\lambda}\otimes \Cal L_{-\lambda})$
to
 $C^{\infty}_{\underline{\Cal D}\times \Cal D} (G_\R\times G_\R/(K_\R\times K_\R),  V^{\vee}_{\mu}\otimes V_{\mu})$.
 \begin{enumerate}
\item Let $\Cal P_{\text{opp}}^h$ be the Hermitian transpose to the Penrose transform
\begin{equation*}
\Cal P_{opp} :  H^{n,s}((G_\R/L_\R)^{\text{opp}}, \Cal L_{-\lambda})\to  C^{\infty}_{\underline{\Cal D}} (G_\R/K_\R, {V}^{\vee}_{\mu}). 
\end{equation*}
That is,
\begin{align*}
\Cal P_{\text{opp}}^h : \big(C^{\infty}_{\underline{\Cal D}} (G_\R/K_\R, {V}^{\vee}_{\mu})\big)^h &\to H^{n,s}((G_\R/L_\R)^{\text{opp}}, \Cal L_{-\lambda})^h\\
f &\to \Cal P_{\text{opp}}^h(f) \text{ where }\\
  \Cal P_{\text{opp}}^h(f) (\omega) & = f( \Cal P_{\text{opp}}(\omega))
\text{ for  each  cohomology class } [\omega ].
\end{align*}
\item Similarly, let $\big({\Cal P^{-1}_{\text{opp}}}\big)^h$ be the Hermitian transpose of $\Cal P^{-1}_{\text{opp}}$.
By definition, if $\eta \in  H^{n,s}(G_\R/L_\R^{\text{opp}}, \Cal L_{-\lambda})^h$
and $F \in  C^{\infty}_{\underline{\Cal D}} (G_\R/K_\R, {V}^{\vee}_{\mu})$, then $\big({\Cal P^{-1}_{\text{opp}}}\big)^h(\eta) (F) = \eta( \Cal P^{-1}_{\text{opp}} F)$.
\end{enumerate}
\end{Def}

\begin{Lem}
The composition 
\begin{equation*}
\big(\Cal P^{-1}_{\text{opp}} \big)^h\circ \Cal P_{\text{opp}}^h : \big(C^{\infty}_{\underline{\Cal D}}(G_\R/K_\R, {V}^{\vee}_{\mu})\big)^h\to
\big(C^{\infty}_{\underline{\Cal D}}(G_\R/K_\R, {V}^{\vee}_{\mu})\big)^h
\end{equation*}
 is the identity map
on $\big(C^{\infty}_{\underline{\Cal D}}(G_\R/K_\R, {V}^{\vee}_{\mu})\big)^h$. Similarly 
$\Cal P_{\text{opp}}^h\circ \big(\Cal P^{-1}_{\text{opp}}\big)^h$ is the identity map on 
 $H^{n,s}(\big(G_\R/L_\R\big)^{\text{opp}}, \Cal L_{-\lambda})^h$.
\end{Lem}

\begin{proof} It is clear from the definitions.
\end{proof}

\begin{Rem} \label{pa1}

\item Using the explicit formula \ref{pa} one can show that
$$T_{\sigma} \Cal P ( w) (x) = \Cal P_{\text{opp}} (\sigma w) (x)$$
where $\sigma$ is the conjugation described in part (2) of Remark \ref{important}
and $T_{\sigma}$  is the conjugate linear isomorphism
from $V_{\mu}$ to $V^{\vee}_{\mu}$.

\end{Rem}

\begin{Thm}\label{leticia}
The space of separately continuous Hermitian pairings on\newline
$\big(C^{\infty}_{\underline{\Cal D}}(G_\R/K_\R, {V}^{\vee}_{\mu})\big)^h$ is $C^{\infty}_{\underline{\Cal D}\times \Cal D} (G_\R\times G_\R/(K_\R\times K_\R),  V_{\mu}\otimes V^{\vee}_{\mu})$.
\end{Thm}

\begin{proof}
An element  $\phi \in C^{\infty}_{\underline{\Cal D}\times \Cal D} (G_\R\times G_\R/(K_\R\times K_\R),  V_{\mu}\otimes V^{\vee}_{\mu})$ defines a linear map 
\begin{align*}
T_{\phi} : \big(C^{\infty}_{\underline{\Cal D}}(G_\R/K_\R, {V}^{\vee}_{\mu}))\big)^h & \to C^{\infty}_{\underline{\Cal D}}(G_\R/K_\R, {V}^{\vee}_{\mu})\\ 
f &\to f(  T_{\sigma}\otimes 1 \; \phi ).
\end{align*}
 We must show that  the  $T_{\phi}$ is continuous.
Let $\omega$   be a representative of the cohomology class
in $[\big(\Cal P\otimes \Cal P_{\text{opp}}\big)^{-1}(\phi)] \in H^{(n,n)(s,s)}(G_\R/L_\R\times (G_\R/L_\R)^{\text{opp}}, \Cal L_{\lambda}\otimes \Cal L_{-\lambda})$.
By (\cite{Vogan08}, Thm 8),  $\omega$ defines a continuous linear map
\begin{equation*}
T_{\omega} : H_c^{0,n-s}(G_\R/L_\R, \Cal L_{-\lambda})\to H^{n,s}((G_\R/L_\R)^{\text{opp}}, \Cal L_{-\lambda}).
\end{equation*}
We argue  that the  continuous composition $\Cal P_{\text{opp}} \circ T_{\omega}\circ \big( \Cal P_{\text{opp}} \big)^h$
is $T_{\phi}$.
Indeed, if $f\in \big(C^{\infty}_{\underline{\Cal D}}(G_\R/K_\R, {V}^{\vee}_{\mu})\big)^h$, then
 \begin{align}\label{calculus}
 T_{\omega}\circ  \big( \Cal P_{\text{opp}} \big)^h (f) (x)& =
  \big( \Cal P_{\text{opp}} \big)^h (f)(\sigma\otimes 1 \;\omega (x, \cdot))\\ \nonumber
  & \text{by  (2) in Remark \ref{important}}\\ \nonumber
  &= f\big(\sigma\otimes 1 \circ \Cal P_{\text{opp}} \;\omega (x , \cdot))\\ \nonumber
  &\text{by definition of $\big( \Cal P_{\text{opp}} \big)^h$}.\\ \nonumber
  \end{align}
 Hence,
\begin{align*}
\Cal P_{\text{opp}} \circ T_{\omega}\circ \big( \Cal P_{\text{opp}} \big)^h(f)  & =
f \big( \Cal P_{\text{opp}}\circ \sigma\otimes 1 \otimes \Cal P_{\text{opp}} \;\omega (\cdot , \cdot))\\ 
  &  = f \big(  T_{\sigma}\otimes 1 \circ \Cal P\otimes \Cal P_{\text{opp}} \;\omega (\cdot , \cdot))\\
  &\text{ by Remark \ref{pa1}}\\
  & = f \big(  T_{\sigma}\otimes 1\; \phi) = T_{\phi}(f).
  \end{align*}

  This shows that $T_{\phi}$ is continuous.
  
 To complete the proof  we show that every continuous linear map \newline
 $T: \big(C^{\infty}_{\underline{\Cal D}}(G_\R/K_\R, {V}^{\vee}_{\mu})\big)^h  \to 
  C^{\infty}_{\underline{\Cal D}}(G_\R/K_\R, {V}^{\vee}_{\mu})$
  is of the form $T_{\phi}$ for some  section  $\phi$ in  $C^{\infty}_{\underline{\Cal D}\times \Cal D} (G_\R\times G_\R/(K_\R\times K_\R),  V_{\mu}\otimes V^{\vee}_{\mu})$. Given  such a map $T$, the composition
  $\Cal P_{\text{opp}}^{-1}\circ T\circ \big(\Cal P^{-1}_{\text{opp}} \big)^h$ is a continuous linear map
  from $H_c^{0,n-s}(G_\R/L_\R, \Cal L_{-\lambda})$ to $H^{n,s}((G_\R/L_\R)^{\text{opp}}, \Cal L_{-\lambda})$.
  By (\cite{Vogan08}, Thm 8), there exists a cohomology class
  $[\omega ] \in  H^{(n,n)(s,s)}(G_\R/L_\R\times (G_\R/L_\R)^{\text{opp}}, \Cal L_{\lambda}\otimes \Cal L_{-\lambda})$
so that $\Cal P_{\text{opp}}^{-1}\circ T\circ \big(\Cal P^{-1}_{\text{opp}} \big)^h = T_{\omega}$.
Hence, $T = \Cal P_{\text{opp}} \circ T_{\omega} \circ \big(\Cal P_{\text{opp}} \big)^h$.
Now, the computation in (\ref{calculus}) shows that $T = T_{\phi}$ for $\phi = \big(\Cal P\otimes \Cal P_{\text{opp}}\big) (\omega)$.
\end{proof}

\subsection{ Hermitian forms on $H_c^{0,n-s}(G_\R/L_\R, \Cal L_{-\lambda})$
in terms of Hermitian forms on $\big (\mathop{Ker} \underline{\Cal D}\big)^h$}\label{hf}

\begin{Lem}
Let $[\omega] \in  H^{(n,n)(s,s)}(G_\R/L_\R\times (G_\R/L_\R)^{\text{opp}}, \Cal L_{\lambda}\otimes \Cal L_{-\lambda})$
and let $\phi =  \Cal P\otimes \Cal P_{\text{opp}} (\omega)$. Then,
\begin{enumerate}
\item $\langle \; , \; \rangle_{\omega}$ is $\text{diag}(G_\R\times G_\R)$-invariant if and only if $\langle \;,\;\rangle_{\phi}$ is $\text{diag}(G_\R\times G_\R)$-invariant.
\item $\langle \; , \; \rangle_{\omega}$ is positive definite if and only if $\langle \;,\;\rangle_{\phi}$
 is positive definite.
 \end{enumerate}
 \end{Lem}
 
 \begin{proof} If $f_1, f_2 \in \big(\text{Ker}(\underline{\Cal D})\big)^h = \big(C^{\infty}_{\underline{\Cal D}}(G_\R/K_\R, {V}^{\vee}_{\mu})\big)^h$, then
 \begin{align*}
 \langle f_1 , f_2 \rangle_{\phi} &= f_1 [T_{\phi}(f_2)]
 = f_1[\Cal P_{\text{opp}} \circ T_{\omega}\circ \big( \Cal P_{\text{opp}} \big)^h (f_2)]\\
 &\text{ by the argument in the proof  of Theorem (\ref{leticia})}\\
 &= (\Cal P_{\text{opp}}^h f_1) [T_{\omega} ( \big( \Cal P_{\text{opp}} \big)^h (f_2))]\\
 &\text{by the definition of  $\Cal P_{\text{opp}}^h$}\\
 & = \langle   \Cal P_{\text{opp}} ^h(f_1) ,  \Cal P_{\text{opp}}^h(f_2) \rangle_{\omega}.
 \end{align*}
 \end{proof} 
 
 \section{A description of the unitary globalization of $A_{\fq}(\lambda)$}
 
 The aim of this section is to describe the unitary globalization of $A_{\fq}(\lambda)$.
 Our assumption on $\lambda$ (\ref{condition}),  guarantee that $A_{\fq}(\lambda)$ is irreducible
 and unitarizable \cite{Vogan84}.  
It follows that the space of $\text{diag}(G_\R\times G_\R)$-invariant  Hermitian forms on the cohomology space
 $H_c^{0,n-s}(G_\R/L_\R, \Cal L_{-\lambda}) = H^{n ,s}((G_\R/L_\R)^{\text{opp}}, \Cal L_{-\lambda})^h$ 
 is one-dimensional. 
 By the results in  sections \ref{hp} and \ref{hf}, in order to identify the unitary globalization of $A_{\fq}(\lambda)$
  it is enough to identify   a $\text{diag }(G_\R \times G_\R)$-invariant 
 section
 $F$  in $C^{\infty}_{\underline{\Cal D}\times \Cal D} (G_\R\times G_\R/(K_\R\times K_\R),  V_{\mu}\otimes V^{\vee}_{\mu})$.

It is well known,  see \cite{Harish-Chandra56},  that  each  admissible $(\fg, K)$-module may be realized as the space of $K$-finite vectors of some Hilbert space globalization.
 Let $(\pi_{\lambda}, \Cal H_{\lambda})$ be a Hilbert space globalization of
 the admissible $(\fg, K)$-module $ \mathop{Ker }(\Cal D)_{K-\text{finite}}$.
 Write  $ (\pi_{\lambda}^{\vee},  \Cal H_{\lambda}^{\vee})$  with $\pi_{\lambda}^{\vee}(g) = \pi_{\lambda}^t (g^{-1})$, 
  the contragredient representation.  
 Let $\Cal H_{\lambda}(\mu)$ be the $K$-isotypic subspace of $\Cal H_{\lambda}$ for the minimal
 $K$-type $\tau_{\mu}$. Write $E_{\mu}$ for the orthogonal $K$-equivariant projection
 of $\Cal H_{\lambda}$ onto $\Cal H_{\lambda}(\mu)$. 
 By Theorem \ref{Wong} the multiplicity of  $\Cal H_{\lambda}(\mu)$ in $\Cal H_{\lambda}$ is one.
  Choose a basis  $\{v_i\}$ of $\Cal H(\mu)$, 
orthonormal with respect to  the Hilbert space inner product $\langle \; , \; \rangle$.  Let $\{ v^*_i\}$ be  the dual basis.
We show that the   generalized spherical function
 \begin{align}\label{spherical}
(x,y) \mapsto 
 &= \big( \frac{1}{\text{dim}(\Cal H(\mu))}\big )\; \sum _i  v_i^*\big[ \pi_{\lambda}(x^{-1} y)  \cdot v_i \big]\\\nonumber
 &= \big( \frac{1}{\text{dim}(\Cal H(\mu))}\big )\; \sum _i   \pi_{\lambda}^{\vee} (x)v_i^*\big[ \pi_{\lambda}( y)  \cdot v_i \big]
 \end{align}
when suitably interpreted, determines  the $\text{diag}(G_\R\times G_\R)$-invariant  Hermitian forms on 
 $\big(C^{\infty}_{\underline{\Cal D}}(G_\R/K_\R, {V}^{\vee}_{\mu})\big)^h.$

 \subsection{A $\text{diag }(G_\R \times G_\R)$-invariant  section of $(G_\R\times G_\R/(K_\R\times K_\R),  V_{\mu}\otimes V^{\vee}_{\mu})$}

 We  interpret the function in (\ref{spherical}) 
 as a smooth section of the bundle $(G_\R\times G_\R)\times_{(K_\R \times K_\R)}(V_{\mu}\otimes V^{\vee}_{\mu})$.
 To accomplish this, we  (a) identify $V_{\mu}$  with $\Cal H(\mu)$ and (b)  realize $\Cal H(\mu)\otimes \Cal H^{\vee}(\mu)$, via Peter-Weyl Theorem, as a submodule of\newline
  $\text{span}_{\C} \{ K_\R\times K_\R \text{ matrix coefficients of  } \Cal H(\mu)\otimes \Cal H^{\vee}(\mu)\}$. Indeed,    the $K_\R\times K_\R$-module $\Cal H(\mu)\otimes \Cal H^{\vee}(\mu)$ is
equivalent to the $K_\R\times K_\R$ representation  acting on 
\begin{equation*}
\text{span}_{\C} \{ (k_1, k_2) \to  \big \langle (k_1\times k_2 ) \cdot \sum_i (v_i \otimes {v_i}^*), v_j\otimes v^*_k
\big\rangle \vert
\; j,k \in \{1, \ldots, \text{dim}(\Cal H(\mu))\}\}.
\end{equation*}
Observe that for $(x, y) \in G_\R\times G_\R$ fixed and $(k_1, k_2) \in K_\R\times K_\R$,
\begin{align*}
 \text{Trace} \big( E_{\mu}\circ \pi_{\lambda}(k_1^{-1}& x^{-1}y k_2) \circ E_{\mu}\big) =\\
& = \sum _i  (k_1\cdot v_i^*) \big[ \pi_{\lambda}(x^{-1} y) k_2 \cdot v_i \big]\\
 & = \sum_i \sum_{j,k} \langle k_2 \cdot v_i, v_j \rangle \; \langle k_1 v_i^*, v_k^*\rangle \; v^*_k\big[ \pi_{\lambda}
 (x^{-1}y) v_j \big]\\
 & = \sum_{j,k} \langle k_1\times k_2 \cdot \sum_i(v_i\otimes v_i^*), v_j,\otimes v_k^* \rangle \; v_k^*(\big[ \pi_{\lambda}
 (x^{-1}y) v_j \big]\\
 & \subset \text{span}_{\C} \{  \langle k_1\times k_2 \cdot \sum_i (v_i\otimes v_i^*), v_j\otimes v^*_k\rangle \\
 &\hskip 2in j,k \in \{1, \ldots, \text{dim}(\Cal H(\mu))\}\}.
\end{align*}

We summarize the above observation in the following Lemma.

  \begin{Lem}\label{section}
  The function
  \begin{equation*}
  F : G_\R\times G_\R \to V_{\mu}\otimes V^{\vee}_{\mu} \subset \mathrm{span}_{\C} \{ \text{matrix coefficients of } V_{\mu}\otimes V^{\vee}_{\mu}\}
  \end{equation*}
  given by
  \begin{equation*}
  F(x,y)(k_1,k_2) = \frac{1}{\dim  (\Cal H(\mu)} \;  \mathrm{Trace} \big( E_{\mu}\circ \pi_{\lambda}(k_1^{-1}x^{-1}y k_2)
 \circ E_{\mu}\big),
\end{equation*}
defines a smooth section of the vector bundle \begin{equation*}
(G_\R\times G_\R)\times_{(K_\R \times K_\R)}(V_{\mu}\otimes V^{\vee}_{\mu}).
\end{equation*}
 \end{Lem}
 
 \begin{Thm}\label{key}
 The section $F$ of the vector bundle $(G_\R\times G_\R)\times_{K_\R \times K_\R}(V_{\mu}\otimes V^{\vee}_{\mu})$
 given in Lemma (\ref{section}) is annihilated by the differential operators $\underline{\Cal D} \otimes 1$ and $1 \otimes \Cal D$.
 \end{Thm}
 
 \begin{proof}
We show that $(1 \otimes \Cal D) (F) =0$ . The proof of the identity $(\underline{\Cal D} \otimes 1) (F) = 0$ is similar.
According to Definition (\ref{operators}), given $\{ X_{\beta}\}$ an orthonormal basis of $\fp$ consisting of root vectors,
we must show that
\begin{equation*}
 (1\otimes \mathbb P) \big[  \sum_{\beta\in \Delta(\fp)} (R(1\otimes X_{\beta}) F) (x,y ) \otimes X_{-\beta} \big] = 0
\end{equation*}
where $\mathbb P$ is the canonical projection  $\mathbb P: V_{\mu}\otimes \fp \to V_{\mu}^-$.
Thus, it is enough to show that for each $\delta \in \Delta(\fu\cap\fp)$,
\begin{equation*}
\int_{K_\R} \tau^{\vee}_{\mu}\otimes \tau_{\mu} (1\times k)\otimes \text{Ad}(k)\;
\big\{ \sum_{\beta\in \Delta(\fp)} \big(R(1 \otimes X_{\beta} )F\big) (x,y) \otimes X_{-\beta} \big\}\; \bar{\chi_{\mu-\delta}(k)} dk = 0,
\end{equation*}
where  $\chi_{\mu-\delta}$ is the character of the irreducible $K_\R$ module with highest weight $\mu -\delta$.
We observe, as the definition of $\Cal D$ is independent of the basis, that
\begin{align}\label{one}
& \int_K  \sum_{\beta\in \Delta(\fp)} \big\{ \big(R \big (1\otimes \text{Ad}(k) X_{\beta} )F \big) (x, y k) \otimes \text{Ad}(k) X_{-\beta} \big \} \; \bar{\chi_{\mu-\delta}(k)} dk\\ \nonumber
& = \int_K  \sum_{\beta\in \Delta(\fp)} \big\{ \big(R \big (1\otimes  X_{\beta}) F \big) (x, y k) \otimes  X_{-\beta} \big \} \; \bar{\chi_{\mu-\delta}(k)} dk.
\end{align}
 On the other hand,
 \begin{equation*}
  R  (1\otimes  X_{\beta}) \;F  (x, y k) = \sum_i v_i^* \big[ \pi_{\lambda}(x^{-1}yk)\;  X_{\beta}\cdot v_i \big].
  \end{equation*}
  As the vectors $v_i$ are $K$-finite and $X_{\beta} \in \fg$, the vector $X_{\beta}\cdot v_i$ is $K$-finite.
  Thus, there exists a finite set $S \subset \hat{K}$ so that 
  \begin{equation} 
  \label{two}
  X_{\beta}\cdot v_i
  = \sum_{\tau \in S} \mathbb P_{\tau} [ X_{\beta}\cdot v_i ].
  \end{equation}
  Replacing identity (\ref{two}) in the displayed formula (\ref{one}) we get
\begin{align}\label{three}  
 & \int_K  \sum_{\beta\in \Delta(\fp)} \big\{ \big(R (1\otimes  X_{\beta} )F \big) (x, y k) \otimes  X_{-\beta} \big \} \; \bar{\chi_{\mu-\delta}(k)} dk =\\ \nonumber
&=    \sum_{\beta\in \Delta(\fp)}  \sum_{i, \tau\in S} \int_K
v_i^*\big[ \pi_{\lambda}(x^{-1}y) \pi_{\lambda}(k) \; \mathbb P_{\tau}[ X_{\beta} v_i] \big ] \otimes X_{-\beta} \;\;
\bar{\chi_{\mu-\delta}(k)} dk.
\end{align}
By Theorem \ref{Wong} the $K$-type $\tau_{\mu-\delta}$ does not occur in $\pi_{\lambda}\vert_K$.
Hence, the right hand side in equation (\ref{three}) is zero.
\end{proof}

\begin{Thm}\label{main}
The generalized spherical function 
\begin{equation*}
F(x,y) =  \frac{1}{\text{dim}(\Cal H(\mu)} \;  \text{Trace} \big( E_{\mu}\circ \pi_{\lambda}(x^{-1}y)
 \circ E_{\mu}\big)
 \end{equation*}
  determines the unique (up to a scalar) invariant Hermitian form on $\text{Ker }(\Cal D)^h$.
 The cohomology class $[\omega]  \in  H^{(n,n)(s,s)}(G_\R/L_\R\times (G_\R/L_\R)^{\text{opp}}, \Cal L_{\lambda}\otimes \Cal L_{-\lambda})$ that corresponds to $F$ via the Penrose transform determines the unique (up to scalar)
 Hermitian form on  $H_c^{0,n-s}(G_\R/L_\R, \Cal L_{-\lambda})$.
 \end{Thm}

\begin{proof} The Theorem follows from combining Theorem \ref{hdual}, Theorem \ref{leticia}, Theorem \ref{key} and Proposition \ref{homeo}.
\end{proof}

 When $\pi_{\lambda}$ is a representation in the discrete series,
$F(x, y) = \psi_{\lambda}(x^{-1}y)$,  the function introduced by Flensted-Jensen in (\cite{Flensted-Jensen80}, (7.11)).
The image of the intertwining map $T_{\psi_{\lambda}}$ is the space
of square-integrable sections in $\text{Ker } \underline{\Cal D}$.

\subsection{Acknowledgment} We would like to thank Vladimir Soucek  for many useful conversations.
The first author thanks Charles Univeristy, Prague and the Max-Planck Institute f\"ur Mathematik for
their warm hospitality. The first author's visit to Charles University was supported by the EC centrum in
Prague through grant P201/12/G028.

\end{document}